\documentclass[12pt]{article}

\newcommand{\comment}[1]{}
\newcommand{\complexes}{{\bf C}}

\providecommand{\qed}{\vrule height 6pt depth 0pt width 3 pt}

\newcommand{\reals}{{\bf R}}

\newcommand{\BibTeX}{{\rm B\kern-.05em{\sc i\kern-.025em b}\kern-.08em     
    T\kern-.1667em\lower.7ex\hbox{E}\kern-.125emX}}

\usepackage{hyperref}

\newcommand{\note}[1]{}
\renewcommand\marginpar[1]{}

\usepackage{amsmath}
\usepackage{amssymb}
\newcommand{\cauchy}{{\cal C}}
\newcommand{\scatter}{{\cal R}}

\newcommand{\dotarg}{\hskip .08em \cdot\hskip .08em} 

\newenvironment{proof}[1][Proof]{\begin{trivlist}\item[\hskip \labelsep
{\it #1. }]}{\hfill \qed \goodbreak \end{trivlist}}   
   
\numberwithin{equation}{section} 
\newtheorem{theorem}[equation]{Theorem}
\newtheorem{proposition}[equation]{Proposition}

\newtheorem{lemma}[equation]{Lemma}
\renewcommand{\thetheorem}{\arabic{section}.\arabic{theorem}}
\renewcommand{\theequation}{\arabic{section}.\arabic{equation}}

\begin{document}
\title{Action of a scattering map on weighted Sobolev spaces in the plane}

\author{
 R.M.~Brown\footnote{
Russell Brown   is  partially supported by a grant from the Simons
Foundation (\#195075).
} \\ Department of Mathematics\\ University of Kentucky \\
Lexington, KY 40506-0027, USA
\and
K.A.~Ott\footnote{Katharine Ott is partially supported by a grant from the
  U.S.~National Science Foundation, 
  DMS 1201104.} 
\\Department of Mathematics\\Bates College\\Lewiston, ME 04240, USA
\and
P.A.~Perry
\footnote{Peter Perry is partially supported by a grant from
  U.S.~National Science Foundation, DMS 1208778.}
\\ Department of Mathematics\\ University of Kentucky \\
Lexington, KY 40506-0027, USA
}

\date{}
 
\begin{abstract}
We consider a scattering map  that arises in the $ \bar \partial $
approach to the scattering theory for the Davey-Stewartson II  equation
and show that the map is an invertible map between certain weighted
$L^2$ Sobolev spaces. 
\end{abstract}
\maketitle

\section{Introduction}\label{Introduction}
This note considers a map ${ \cal R }$ that arises in the scattering
theory for a first-order system (\ref{dbar}) in the plane. The map $ {\cal R}$
takes a potential $q$ to scattering data that appears as a coefficient
in the asymptotic expansion of solutions to the first-order system.
The map $\scatter$ also arises in the so-called $ \bar \partial$
approach to scattering theory for the Davey-Stewartson II equation.  The
scattering transform we discuss below first arose in work of Fokas and
Ablowitz \cite{FA:1984} as a transformation
to convert one of the Davey-Stewartson equations to a linear evolution
equation. This same transform also proved to be useful in the work of
Brown and Uhlmann \cite{BU:1996} and Barcel\'o, Barcel\'o and Ruiz
\cite{BBR:2001} who investigated the inverse conductivity problem with
less regular coefficients in planar domains.

The formal theory and the beginning of a rigorous treatment of the
scattering transform may be found in work of Fokas and Ablowitz
\cite{FA:1984} and Beals and Coifman \cite{BC:1988}. One remarkable feature
of Beals and Coifman is that the map ${\cal R}$ satisfies
a Plancherel identity, 
$$
\int _ { \complexes} |q|^2 \, dx = \int_{\complexes} |{\cal R}(q) |^2
\, dk, 
$$
at least for potentials $q$ that  are sufficiently regular. Since the
map ${\cal R}$ is not linear, this identity does not imply the
continuity of the map ${\cal R}$  or even that $ \scatter (
q) $ is  defined for all $q$ in $L^2$.  
There are several authors who have established continuity of the
transform on other spaces.   Sung
\cite{LS:1994a,LS:1994b,LS:1994c} 
 develops estimates for the scattering transform in the Schwartz space. Brown \cite{RB:2001b} establishes that the map ${\cal R}$ is continuous in a
neighborhood of  the origin in $L^2$. Perry \cite{PP:2011a} considers the map on a
weighted Sobolev space $H^{ 1,1}$ and shows that $ {\cal R}$ is
locally Lipschitz continuous on this space.  The weighted Sobolev spaces   $H^ {
  \alpha, \beta}$ are defined as 
$$
H^ { \alpha, \beta } = \{ f : \langle D \rangle ^ \alpha f \in L^2
\mbox{ and }  \langle \cdot
\rangle ^ \beta  f  \in L^2\}.$$
We use $ \langle x \rangle $ to denote $ \langle x \rangle = ( 1+|x|^2
) ^ { 1/2}$ and then $ \langle D\rangle^\beta$ is the Fourier multiplier
operator $\langle D \rangle ^ \beta f = ( \langle \dotarg \rangle ^
\beta \hat f ) \check {}$. Astala, Faraco, and Rogers \cite{MR3382582}
consider the map ${\cal R}$ on the space $ H^ { \alpha, \alpha} $ with
$ \alpha >0$ and show that ${\cal R } : H^ { \alpha, \alpha}
\rightarrow L^2$ is locally Lipschitz continuous. The main result of
the current work is to show that $ {\cal R } : H^ { \alpha, \beta
}\rightarrow H ^ { \beta, \alpha}$ when $ 0 < \alpha, \beta < 1$ and
that the map ${ \cal R}$ is locally Lipschitz continuous.   We also
give a result which shows that in the spaces $H ^ {\alpha,
  \beta}$, the difference $ \scatter (q) -\hat q$ is better
behaved than $\hat q$. In particular, if $q$ is in $ H ^ { \alpha, \beta}$, then $
\scatter (q) - \hat q$ will lie in $ H ^ { 2 \beta, 2\alpha}$, at
least for $ \alpha, \beta < 1/2$. 

Our  result in this  paper is a
two-dimensional analogue of the results of X.~Zhou 
\cite{MR1617249}  which give mapping properties of a scattering
transform for the ZS-AKNS system on weighted Sobolev spaces on the real line.

We begin our development by sketching the definition of the map ${\cal
  R}$. We let $q$ be a function on the complex plane and for much of
the argument we will assume that $q$ is in the Schwartz class
${\cal S} (\complexes)$. We will establish estimates on the map $ {\cal
  R}$ with constants that depend only on the norm of $q$ in a weighted
Sobolev space. With these estimates it will be possible to extend the
map $ {\cal R} $ from the Schwartz space to $ H ^ { \alpha,
  \beta}$ with $ \alpha >0$ and $ \beta >0$.  Throughout the paper, we
 will use $ e _k (x) = e_x(k) = \exp ( \bar k \bar x - kx
)$.  We consider the system 
\begin{equation}
\label{dbar}
\left \{
\begin{aligned}
&\bar \partial \mu _ 1(x,k)= \frac 1 2 e_k(x) q(x)  \bar \mu _2 (x,k) \\
&\bar \partial \mu _2 ( x,k) = \frac 1 2 e_k (x) q(x) \bar \mu _1
(x,k) \\
&\lim _ { |x| \rightarrow \infty } ( \mu _ 1(x,k) , \mu _2 (x,k) ) = (
1,0).
\end{aligned}
\right. 
\end{equation}
Above, 
 $ \bar \partial = \frac 1 2 (\frac \partial { \partial x_
  1}  + i \frac \partial { \partial x _ 2 } ) $ and 
$  \partial = \frac 1 2 (\frac \partial { \partial x_
  1}  - i \frac \partial { \partial x _ 2 } ) $   denote the
standard derivatives  with respect to the complex variables $ \bar
x$ and $x$. When we need to differentiate with respect to $ k$ and $
\bar k$, we will write $ \partial / \partial k $ and $ \partial /
 \partial \bar k$. 

We define our scattering transform  $ {\cal R}$ by 
$$
{\cal R }( q) (k) =  \frac 1 \pi \int _ \complexes e _k (x) q(x) \bar \mu _ 1 ( x
,k ) \, dx
.
$$
Since the function $ \mu _1$ approaches one at infinity, it is
plausible that $ {\cal R } $ is a non-linear generalization of the
Fourier transform. To make this more precise, we  introduce a variant
of the Fourier transform that we will use throughout the paper. 
For 
a function $ \psi $ in $  L^1$, we define our Fourier transform by 
\begin{equation*}
\begin{split}
\hat \psi (k) & = \frac 1 \pi \int _ {\complexes } e _k (x) \psi (x) \,
dx \\
& = \frac 1 \pi \int _ {\complexes } e^ { -2i (k_1 x _2 + k_2 x_1)} \psi (x) \,
dx  . 
\end{split}
\end{equation*}
From the second expression it is clear how our Fourier transform is
related to more common normalizations of the Fourier transform. 
With our definition, the   Fourier transform is the linearization at 0
of the scattering transform.  For convenience, we list several
standard properties of the Fourier transform  translated to our
normalization. If we put
$$
\check u ( x ) =   \frac 1 \pi \int _ { \complexes } e _k (-x) u(k) 
\, dk , 
$$
 the Fourier inversion formula reads 
$$
\check { \hat u } = \hat { \check u } = u,
$$
at least for $u$ in the Schwartz class. If we let $ f* g$
denote the convolution,  we have that 
\begin{equation} 
( f*g) \hat (k ) =  \pi \hat f(k) \hat g(k) 
\end{equation}
and 
\begin{equation}
( fg ) \hat (k) =  \frac 1 \pi( \hat f * \hat g) ( k) .
\end{equation}
We recall that the Cauchy transform, 
$\cauchy$, defined by 
$$
{\cal C}f(x) = \frac 1 \pi \int _ { \complexes } \frac { f(y) } {x- y  } \,
dy 
$$
gives a right inverse to the operator $ \bar \partial$. 
Using that $ \widehat {( \bar \partial f ) } (k) =  \bar k \hat f (k)
$,  $ \widehat {( \partial f ) } (k) = - k \hat f (k) $ and the
representation of $ \bar \partial ^ { -1}$ and a similar
representation of $ \partial ^ { -1}$, 
we obtain that  $ (1/z)\hat {} = 1 / \bar k$ and $ (1/\bar z )\hat {} =
-1/k$. From these observations, we obtain
\begin{align}
\label{C1} \frac 1 \pi \int _ { \complexes } \frac { e _k (y) q(y) } {
   \bar y - \bar x } \, dy & = \frac 1 \pi \int _ { \complexes } 
\frac { \hat q (\ell ) e _ x ( k - \ell) } { k - \ell } \, d\ell \\
\label{C2} \frac 1 \pi \int _ { \complexes } \frac { e _k (-y) \bar q(y) } {
    y -  x } \, dy & = \frac 1 \pi \int _ { \complexes } 
\frac { \bar { \hat q} (\ell )  e _ x (  \ell-k) } { \bar k - \bar
  \ell } \, d\ell  . 
\end{align}
These formulae will be used in section \ref{bl} below. 

We are ready to outline the construction of the solutions of the
system (\ref{dbar}). 
We let $ T_k$ be the operator given by 
$$
T_k f(x) = \frac 1 2 \cauchy ( e _ k( \cdot) q \bar f )(x) 
$$
where $ q $ is the potential. Throughout this paper, we assume that $
q$ is in the Schwartz class in order to simplify the argument. As a
last step, we will use the local Lipschitz continuity of $\scatter$ 
from $ H^ { \alpha, \beta }$ to $ H^ {\beta , \alpha }$ to 
extend the map to a weighted Sobolev space. 
It   is clear that if $q$ is in $ H^ { \alpha, \beta}$ with $ \alpha
>0$ and $ \beta >0$,  $ \mu_1 $ and $ \mu
_2$ are in $L^ \infty ( \complexes ^2)$  and  are solutions of the integral equations 
\begin{align} 
\label{IE1A} \mu _1 & = 1 + T_k ( \mu_2)  \\
\label{IE1B} \mu _2 & = T_k ( \mu _1 ) 
\end{align}
then $ ( \mu _1 , \mu _2) $ are solutions of  (\ref{dbar}). 
Furthermore, if we substitute
(\ref{IE1B}) into (\ref{IE1A}) we obtain 
\begin{equation} 
\label{IE2} \mu _1 = 1 + T_k ^2 ( \mu _1) 
\end{equation}
and if $ \mu _1 $ is a solution of  (\ref{IE2}), then $( \mu _ 1 , T_k ( \mu _1 )
) = ( \mu _1, \mu _2) $  is a solution of  (\ref{dbar}). 

Finally, we observe that the map $ \scatter $ is invertible and the
inverse map $ {\cal I }$ can  be given by 
\begin{equation} 
\label{inverse}
 {\cal I }(f) = \overline {
  \scatter ( \bar f )}.
\end{equation}
The invertibility of $ \scatter $ on the Schwartz space is in the work
of Sung \cite{LS:1994a,LS:1994b,LS:1994c} and Perry \cite{PP:2011a}
gives invertibility on $H^ { 1,1}$. The estimates of 
this paper will allow us to extend the invertibility to the family of
spaces $ H^ { \alpha, \beta}$. The formula (\ref{inverse}) can be
found in the work of 
 Astala, Faraco, and Rogers \cite{MR3382582} and Perry. 

Our main result is the following theorem which gives the properties of
the scattering map.
\note{Revise to include information about the inverse map. } 
\begin{theorem}\label{scattered} 
The map $ { \cal R } $ maps $H^ { \alpha, \beta } $ to 
$ H ^ { \beta , \alpha}$ and is locally Lipschitz continuous, provided
$ 0 < \alpha, \beta < 1$. 

More precisely, if we fix $ \alpha$ and $ \beta$  in $ (0, 1)$, then
there exists an increasing function  in $ M_0$,   $ C = C( M_0)= C(M_0,
\alpha, \beta)$ so that if $\| q\|_ {H^{ \alpha, \beta } } \leq M_0$ and 
$\| q'\|_ {H^{ \alpha, \beta } } \leq M_0$  then 
$$
\| \scatter ( q ) - \scatter ( q') \| _ { H ^ {\beta,  \alpha} }
\leq C \| q -  q' \| _ { H ^ {\alpha, \beta}}.
$$

In addition, if $ \alpha , \beta $ are in the interval $(0, 1/2)$,
then 
$$
\| \scatter ( q) - \hat q \| _ { H ^ { 2\beta , 2\alpha }} \leq C (
M_0), 
\qquad \mbox{if } \| q \| _ { H ^ { \alpha, \beta } } \leq M _0. 
$$

Thanks to the identity (\ref{inverse}), the same results hold for the
inverse scattering map. 
\end{theorem}

We begin the proof of Theorem \ref{scattered} by  iterating the
integral equation (\ref{IE2}) 
for $ \mu _1 $ to obtain  finite  expansions for $\mu _1$, 
\begin{equation}
\label{exp} 
\mu _1 (x,k) = \sum _{ j = 0 } ^ { N-1} T_k ^ { 2j } ( 1) ( x) 
+ T _ k ^ { 2N } ( \mu _1 ) (x), \qquad N = 1, 2, 3, \dots .
\end{equation}
We substitute this expression for $ \mu _1$ into the definition of
$\scatter ( q) $ to obtain 
\begin{equation}
\label{rexp}
\scatter (q) (k) = \sum _ { j =0 } ^ { N-1} r _ j (k) + r ^ { (N) }
(k) , \qquad N = 1, 2, 3, \dots,
\end{equation}
where 
\begin{equation} \label{multdef}
r_ j (k ) =   \frac  1 \pi \int _ { \complexes } q (y) e _ k (y)
\bar T _ k ^ { 2j} ( 1 ) (y) \, dy 
\end{equation}
and the remainder term is given by 
\begin{equation} \label{remdef}
r ^ { (N)}(k) =  \frac 1 \pi \int_\complexes q ( y) e _ k (y) \bar T _ k ^ { 2N}
( \mu_ 1 ( \cdot, k) )(y) \, dy .
\end{equation}
The notation $ \bar T _k ^ { 2N} (f)$ is ambiguous and we intend  $
\bar T_k ^ { 2N}(f)$ to mean the complex conjugate of $T_k ^ {2N}(f)$, 
$ \bar T ^ { 2N} _k (f)  = \overline{ ( T_k ^ { 2N} (f) ) } $. 

The term $ r _0 = \hat q$ is just the Fourier transform. 
We will use duality and estimates for certain Brascamp-Lieb forms to
estimate the terms $ r _ j$ for $ j\geq 1$.  For $N$ sufficiently
large (depending on $ \alpha$ and $ \beta$) we will be able to show
that the remainder term is in $H^ { 1,1 }$. This second step is where
we require that $ \alpha $ and $ \beta $ be positive, while the
estimates for the terms  $r_j$ hold  for $ \alpha =0$ or $ \beta =
0$. 
The estimates in $L^2$ can be found in the work of one of the
authors
\cite{RB:2001b}.  Work of Nie \cite{MR2754810} and M. Christ's appendix to
Perry's paper \cite{PP:2011a} give different proofs of these estimates
for the terms $r_j$. 

Our  argument follows the argument of Perry 
 when the potentials  are in $H^ { 1,1 }$.  The innovations in this
 paper are 
new estimates for multi-linear forms and  a certain amount of
persistence that  is needed to estimate the remainder term.

The outline of this paper is as follows. In section 2, we give more
details on the construction of the special solutions $(\mu _1 , \mu
_2)$. In section 3, we prove the estimates for the multi-linear forms in (\ref{multdef})
and in section 4, we study the remainder term (\ref{remdef}).

{\em Acknowledgment. }We thank K.~Astala, D.~Faraco, and K.~Rogers
for showing us a preliminary version of their work \cite{MR3382582}. This was helpful 
in carrying the research reported below. In particular, their decay
estimate (\ref{AFR}) below is a main step in our work. 

\section{Constructing the scattering solutions}

In this section, we collect several estimates that will be needed
in section \ref{best}
 and give a construction of the solutions  $ (\mu _1, \mu_2)$
to the equations (\ref{dbar}). 

Much of our analysis will take place in $L^p$-spaces and we begin by
observing that  for $  \beta > 0$ and $ 1 > \alpha >0$, we have the inclusion
\begin{equation}
\label{embed}
H^ { \alpha, \beta} \subset L^p, \qquad \frac 1 2 - \frac \alpha 2
\leq \frac 1 p < \frac  1 2 + \frac \beta  2 .
\end{equation}
The estimate for $ p >2$ follows from the Sobolev embedding theorem
and the estimate for $p<2$ follows from the inequality of H\"older. 
We let $ I_1$ denote the fractional integral given by
$$
I_1 (f)( x) = \frac 1 \pi \int _ { \complexes } \frac { f(y)}{ |x-y|}
\, dy .
$$
For estimates involving the size of $\cauchy (f)$, the 
 elementary inequality $ |\cauchy (f) (x) | \leq I_1 (|f|)(x)$
means that it is sufficient to give estimates for $I_1 (f)$.
For  $ p\neq 2$, it is well-known (see Vekua \cite{IV:1962})  that
\begin{equation}
\label{vekua}
\| I_1(f) \| _ { L^ \infty } \leq C _p \| f\|^{1/2} _{ L^p } \| f\|^{1/2}_{L^{
    p'}} .
\end{equation}
Above, $p'$ is the usual conjugate exponent.
We recall the well-known Hardy-Littlewood-Sobolev estimate for fractional
integration which may be found in \cite{ES:1970}, for example. If $ 1 <
p < 2$, we have 
\begin{equation}\label{HLS}
\| I_1 (f) \|_ { L^ {\tilde p } } \leq C_p \|f\| _ { L^ p } , \qquad
\frac 1 { \tilde p } = \frac 1 p - \frac 1 2 .
\end{equation}
We will find it useful to work in the weighted $L^p$-spaces, 
$L ^ p _ \alpha = \{ f : \langle  \cdot \rangle ^ \alpha f \in L^ p\}$.
Occasionally, we will also use the scale-invariant or homogeneous version
of these spaces $ \dot L ^ p _ \alpha = \{ f : |\dotarg|^ \alpha f \in
L ^ p \}$. 
For $ 1 < p \leq \tilde p< \infty$ and $\alpha $, $\beta$ satisfying $
-\frac  2 { \tilde p } < \beta \leq \alpha < \frac 2 { p '}$, and
$ \frac  1 {\tilde p} = \frac  1 p - \frac { 1 - \alpha + \beta } 2 $,
we have 
\begin{equation}
\label{SW} 
\| I_ 1(f) \| _ { L^ { \tilde p } _ \beta } \leq C (p, \tilde p ,
\alpha, \beta ) \| f\| _ { L ^ p _ \alpha }.
\end{equation}
The estimate (\ref{SW})
follows easily from the work of Sawyer and Wheeden \cite[Theorem
  1]{MR1175693}. 
The corresponding estimate in the homogeneous spaces $\dot L^ p _\alpha$
is due to Stein and Weiss \cite{MR0098285}. 
Finally, we recall a result of Astala, Faraco and Rogers who show 
that for $\alpha \geq 0$    and  $ 2 < \tilde p < \infty $, we have 
\begin{equation}
\label{AFR}
\| T_k ^ 2 (f) \| _ { L ^ {\tilde p }}
\leq C \langle k \rangle ^ {
  -\alpha }
\|q\| _ { H^ { \alpha, 0  } }^2\| f\| _ { L ^{ \tilde p}}.
\end{equation}
Occasionally in the sequel, we will want to display the dependence of
$T_k$ and $ \mu_j$ on the potential $q$. We will do this by writing $
\mu _ j ( q ; \cdot, \cdot)$ and $ T_{ k, q}$. 

\begin{proposition} 
\label{MUEX}
Let $ \epsilon >0  $ and suppose that $ q \in H^ { \epsilon,
  \epsilon}$. Fix $ \tilde p_0$ with $ 1/\tilde p_0 \in (0, \epsilon /2)$. We
may construct solutions $ (\mu_1, \mu_2)$ of (\ref{dbar}) with $ ( \mu
_1 -1, \mu_2) $ in $L^ \infty _k ( L^ { \tilde p_0 } _x( \complexes
^2) )$. Furthermore, there exists a constant $C  = C( M_0) $ so that if  $ \| q
\| _ { H^ { \epsilon , \epsilon }} \leq M_0 $ and $ \| q' \|_{ H^ {
    \epsilon, \epsilon } } \leq M _0 $, then 
$$
\sup _ k \| \mu_1 ( q; \cdot, k ) - \mu _1 ( q'; \cdot, k) \| _ { L ^
  { \tilde {p_0} } } \leq C(M_0 ) \| q - q' \| _ { H ^ { \epsilon ,
    \epsilon }}.
$$
\end{proposition}
\note{ Can we find an estimate for $ C(M_0)$? I conjecture that our
  proof gives that $ C( R_0) \leq C e ^ {C R_0}$. }

\begin{proof}
We may assume that $ \epsilon \in (0, 1) $. 
 From H\"older's inequality and the
  Hardy-Littlewood-Sobolev inequality (\ref{HLS}) it follows that if
  $q$ is in $L^2$ and  $ \tilde p  $  is  in $(2, \infty)$,  the map $
  f \rightarrow T_k f$ is bounded on $ L^ { \tilde p}$ and we have the
  continuity result
$$
\| T_{ k, q} (f) - T_ { k, q'} (f) \| _ { L^ { \tilde p }}
\leq C \| q - q' \| _ { L^ { 2  }} \| f\| _ { L^ { \tilde p }}. 
$$
Furthermore, by approximating $ q$ in $L^2$ by functions that are
bounded and compactly supported, we can see that the map $T_k$ is
compact on $L^{ \tilde p}$. 
If we also have that $ \langle \dotarg \rangle ^ \theta q \in L^2$, then 
\begin{equation}
\label{OpDecay}
\| T_{k,q} - T_ { k', q } \| _{ {\cal L } ( L ^ { \tilde p } ) }
\leq C | k - k ' |^ \theta \| q\| _ { H ^ { 0, \theta } }, \qquad
\tilde p \in ( 2, \infty ) . 
\end{equation}
To establish (\ref{OpDecay}), observe that 
$$
|e_k (y) - e_ { k' } (y) | \leq 2 ^ \theta |k-k'|^ \theta |y |^
\theta
$$
and thus
\begin{align*} 
| T_k (f) (x) - T_{ k'} (f) (x) |  &\leq \frac 1 \pi  \int _ {
  \complexes } \frac { |f(y) | |q(y) | } { |x-y| }| e_k (y) - e_ { k '
  } ( y)| \, dy  \\
&\leq C |k-k'|^ \theta I _ 1 ( |\cdot| ^ \theta |q||f| ) (x) .
\end{align*}
Thus by the inequality of Hardy-Littlewood-Sobolev  (\ref{HLS}) and
H\"older's inequality, we have  
$$
\| T _ k (f)  - T _ { k ' } (f) \|_ { L^ { \tilde p }}\leq C |k - k '
|^ \theta \| f \| _ { L^{\tilde p }} \| q \| _ { H ^ { 0, \theta } } .
$$
In the argument below, we will apply this estimate with $  \theta =
\epsilon/2$. 

We observe that a solution of (\ref{dbar}) should also solve the
integral equations (\ref{IE2}).  According to a standard argument
using the Liouville theorem for pseudo-analytic functions (see
\cite[Section 3]{BU:1996} or
\cite{LS:1994a}), the
operator $(I - T_k ^2 )$ is injective on  $ L^ { \tilde p } $ for $ 2<
\tilde p < \infty$. Since $ T_k ^2 $ is compact, the Fredholm theory
tells us that $ ( I - T_k^2 ) ^ { -1}$ exists as an operator in $ {
  \cal L } ( L^ { \tilde p })$. The continuity of the map $ (k, q )
\rightarrow T _ { k,q } $ implies that   $\|(I- T_{ k,q } ^ 2 ) ^ {
  -1}  \| _ { { \cal L}( L^ { \tilde p } ) }$ is a continuous
function for $ ( k, q )$ in $\complexes \times H ^ { 0, \epsilon /2
}$. According to the decay estimate of \cite{MR3382582} (restated as (\ref{AFR})), given $ M _0$,
we may find $ R_0 $ so that $ \| T_k ^ 2 \|_{{ \cal L } (L^ { \tilde p }) }
  \leq  1/2$ if $ |k | \geq R _0$ and $ \| q \| _ { H ^ {
      \epsilon , \epsilon } } \leq M_0$.  The embedding $H^ {
    \epsilon, \epsilon } \subset H ^ { 0, \epsilon/2}$ is compact and
  thus the continuous function $ \| ( I - T_k ^ 2) ^ { -1} \|_{ {\cal
      L } ( L^ { \tilde p })} $ will be bounded on the compact set,
$\{ (k, q ) : |k | < R_0 , \| q \| _ { H ^ { \epsilon , \epsilon
  }}\leq M_0 \}  \subset \complexes \times H^{0,\epsilon/2}$. 

If we  use the
embedding of $ H^ {\epsilon, \epsilon }$, (\ref{embed}), and
(\ref{HLS}) we can see that 
$ T_k (1) $ lies in $ L^ \infty _k ( L^  { \tilde p _0
}_x)$. 
Thus we may set 
$$
\mu _1 ( \cdot, k ) = 1 + ( I - T_k ^ 2) ^ { -1} ( T_k (1) ) , \qquad
\mu _ 2 ( \cdot , k ) = T _k ( \mu _1 ) 
$$
and then  $  ( \mu _1, \mu _2) $ is a solution  of 
(\ref{dbar}). 

Next we observe that thanks to the continuity of the map 
$
(k, q) \rightarrow T _ { k, q } ( 1) $ from $ H^ { \epsilon , \epsilon
} $ into $ L^ { \tilde p _0}$, it follows that the map $ q \rightarrow
\mu _ 1( q; \cdot , k ) -1 $ is continuous from $ H ^ { \epsilon ,
  \epsilon } $ into $ L^ { \tilde p _0 }$.  In fact, we can write 
\begin{multline*}
\mu _ 1 ( q; \cdot, k ) - \mu _1 ( q' ;\cdot , k) 
\\ 
= ( I - T_ { k,q} ^ 2 ) ^ { -1} ( T _ { k,q} ^ 2 - T_ { k,q'} ^ 2 ) (1) 
+ (( I - T_ { k,q } ^ 2 ) ^ { -1} - ( I - T_ { k,q' } ^ 2 ) ^ { -1} )
( T_ { k, q' } ( 1) ), 
\end{multline*}
and conclude that 
$$
\| \mu _ 1( q; \cdot, k ) - \mu _ 1 ( q' ; \cdot, k ) \|_ { L^ {
    \tilde p _ 0 }} 
\leq C(R_0 ) \| q - q ' \| _ { H ^ { \epsilon, \epsilon } } , \quad
\mbox{if } \| q \| _ { H ^ {\epsilon, \epsilon } } \leq M _0 \mbox{
  and } \| q' \| _ { H ^ { \epsilon, \epsilon }} \leq M _ 0. 
$$
\end{proof}

Now that we have $ \mu _1$ with  $ \mu _1 -1 \in   L^ \infty _k ( L ^ {
  \tilde p _0})$, our next step is to show that we have $ \mu _1 -1$ in
an interval of $ L^ { \tilde p }$ spaces. 
However, before we do this, we
give a simple lemma that will be used to obtain the local  Lipschitz
continuity of multi-linear expressions.

\begin{lemma} \label{mlllc} Let $ \Lambda : X_0 \times X_1 \times \cdots \times X_N
  \rightarrow Y$ be a bounded  multi-linear operator with $ X_j$ and $Y$
  normed vector spaces. More precisely, assume that for some constant
  $ C_0$,  we have 
$$
|\Lambda ( q_1, \dots  ,  q_N ) | \leq C_0 \prod _ { j =1 } ^ N \|q _j \|_ {
  X_j } .
$$
Then we obtain 
 the following local Lipschitz
  continuity result. If $ \| q _j \|_{X_j} \leq M _0 $ and $  \| q '
  _j \| _ { X_j } \leq M_0$, $ j= 1, \ldots, N $, then 
$$
\| \Lambda ( q_1, \dots , q_N) - \Lambda ( q_1',\dots, q_N') \|_{ Y } 
\leq C_0  M_0^ { n-1} \sum _ { j =1 } ^ N \| q _ j - q ' _ j \| _ { X_j } .
$$
\end{lemma}

\begin{proof} We write 
$$
 \Lambda ( q_1, \dots , q_N) - \Lambda ( q_1',\dots, q_N') 
= \sum _ { j = 1 } ^ N \Lambda ( q_ 1, \dots , q_{i-1} , q _ i - q
_i',q'_{i+1}, \dots, q_ N ' ) 
$$
and use the boundedness of  $ \Lambda $. 
\end{proof}

\note{ Need $ 1/2 > \epsilon >0$. (done)

Need to give estimate for differences.  } 

\begin{proposition}
\label{LPREG}
Let $\epsilon $ lie in the interval $(0, 1)$. 
If $ q , q'\in  H ^ { \epsilon, \epsilon } $ with  $ \| q \| _ {H^ {
    \epsilon , \epsilon } } \leq M _ 0 $ and $ \| q' \| _ { H ^ {
    \epsilon , \epsilon } } \leq M_0$, then $ ( \mu_1 -1 ) \in L^ {
\tilde   p } $ when   $1/ \tilde p$ in $ [0, \epsilon )\cap [ 0, 1/2) $ and 
$$
\| \mu _ 1 ( q ; \cdot, k ) - \mu _ 1 ( q' ;  \cdot, k ) \| _ { L^ {
    \tilde p}}\leq C (M _0, \tilde p ) \|q- q' \| _ { H ^ { \epsilon,
    \epsilon }}.
$$
\end{proposition}
\begin{proof}  Let $ \mu _ 1 = \mu _1 ( q;\cdot, \cdot)$. 
We have $ \mu _1 -1 \in L ^ { \tilde p_0 }$ with $ \tilde p _0 $ as in
Proposition \ref{MUEX}  and then (\ref{IE2}) implies that 
$ \mu _1 - 1 = T_k ^ 2 ( \mu _1 - 1 ) + T_k ^ 2 (1).
$ Since $ H ^ {\epsilon, \epsilon } \subset L^p$, for $ 1/p \in 
[ 1 /2 - \epsilon /2 , 1 /2 + \epsilon /2) $, it follows from the
  inequalities of H\"older and Hardy-Littlewood-Sobolev (\ref{HLS}) that $ T _ k ^ 2 ( \mu _1 - 1
  ) \in L^ { \tilde p } $ for $ 1/ \tilde p  \in ( 0 , \epsilon + 1 /
    \tilde p _ 0 )\cap ( 0, 1/2)  $,   
 $ T _ k ^ 2 ( 1) \in L ^ { \tilde p } $ for $
    1/ \tilde p \in ( 0, \epsilon ) $ and we have the estimates
\begin{equation}
\label{quant}
\| T _ k ^ 2 (\mu _1 - 1) \| _ { L^ { \tilde p }} \leq C \| q \|^2 _ { H
  ^ { \epsilon , \epsilon } }  \| \mu _ 1 - 1 \| _ { L^ { \tilde p _
    0}}, \qquad
\| T _k ^ 2 (1) \| _{ L^ { \tilde p }} \leq C \|q\|^2_ {H^ { \epsilon,
    \epsilon }}. 
\end{equation}
Thus $ \mu _ 1 - 1 \in L ^ {
        \tilde p } $ for $ 1/ \tilde p \in ( 0, \epsilon )$.   Similar
      considerations and the estimate (\ref{vekua}) give that $ \mu
      _1$  is bounded.  

Once we recognize that $ T_k^2 (f)$ is a multi-linear expression in $q$
and $f$, the estimate for the differences follows from (\ref{quant}),
Proposition \ref{MUEX},  and Lemma \ref{mlllc}.  
\end{proof}


\section{Brascamp-Lieb Forms} \label{bl}

We consider a family of Brascamp-Lieb forms. A criterion for the
finiteness of these forms on families of $L^p$-spaces was given by
Barthe \cite[Proposition 3]{FB:1998} and simpler proofs of his criterion were
given by Carlen, Lieb, and Loss \cite[Theorem 4.2]{MR2077162}  and
Bennett, Carbery, Christ, and Tao \cite[Remark 2.1]{BCCT:2010}. We
note that Barthe and  Carlen, Lieb, and Loss also  give information about 
 the best constant in these inequalities. We are not able
to make use of this information in our work. A simple approach to
the finiteness of these forms  can be found in the dissertation
of Z. Nie, see \cite{MR2754810} (and may be  well-known). 

To describe the forms we will consider, fix $N$ and  let $ E \subset \reals ^ {
  N +1}$ be a finite collection of non-zero vectors. The set $E$
carries the structure of a matroid (see \cite{MR2036957} for the
definition). Using $E$, we
will define a closed convex subset of $ \reals ^ { N+1}$ which is
called the matroid polytope for $E$. 
 We denote elements of $  [ 0, 1] ^ E$ as
functions $ \theta : E \rightarrow [ 0, 1] $. If $ A \subset E$ is a
set, then we let $ \chi _A$ be the indicator function of the set
$A$. Thus 
$$
\chi _A ( v) = \left \{ \begin{array}{ll} 1, \qquad & v \in A \\ 0,
  \qquad & v \notin A . 
\end{array}\right. 
$$
The { \em matroid polytope }of $E$, $ { \cal P } (E) $, is defined to
be 
the convex hull of the set $ \{ \chi _B : B \subset E  \mbox { and $B$
  is a basis for $ \reals ^ { N + 1}$} \}$. Given a set of vectors $E$,
  we define a multi-linear form 
$$
\Lambda ( f _ v | v \in E ) = \int _ { \complexes ^ { N +1}} \prod _ {
  v \in E } f _ v ( v\cdot x ) \, d x
$$
where  $ x = ( x_0, x_ 1, \dots , x _N )$ is a point  in  $
\complexes ^ { N +1 }$ and $ v\cdot x = \sum _ { i = 0 } ^ N v_ i x_i$
is the standard bilinear inner product.   We initially assume that the
functions $ f _ v$ are non-negative so that the integral defining $
\Lambda ( f _ v| v \in E )$ will exist, though it may be infinite.
The estimates of Theorem \ref{forms} will give us estimates for the
form when the functions $\{ f_v : v \in E\}$ belong to certain Lorentz spaces. 
These estimates will allow us to extend the form to 
appropriate products of Lorentz spaces.

\note{range of p and r in $L^ { p, r}$, probably want $1<p< \infty$
  and  $ 1 \leq r \leq \infty$. }

We will need to consider these forms when the functions $ f _v$ are in
Lorentz spaces $ L ^ { p, r } ( \complexes )$, $1\leq p,r \leq
\infty$, 
and refer the reader to
the monograph of Bergh and L\"ofstr\"om for the definition
of these spaces\cite[p.~8]{BL:1976}. Our main estimate for
Brascamp-Lieb forms is the following theorem.
\begin{theorem}\label{forms} 
Suppose that the function  $( 1/ p_v ) _ { v \in E} $ lies in the
interior of the set $ { \cal P } ( E_1)$ and $ \sum _ {v \in E } 1/ r _
v \geq 1 $. Then there is a finite constant $ C $ such that 
$$
\Lambda ( f _ v | v \in E ) \leq C \prod _ { v \in E } \|f _ v \| _ {
  L ^ { p_ v , r _ v} }.
$$
\end{theorem}
\note{ We probably need to work through the proofs again to check that
  we don't need to exclude cases such as one vector is a multiple of
  another. } 

\begin{proof} As noted above, this is essentially a result of
  Barthe. To connect our statement to the result of Barthe, we first
  observe that the finiteness of the form on $ \complexes ^ { N +1}$
  is equivalent to the finiteness of the form on $ \reals ^ { N +1 }$,
  {\em i.e. }the form obtained when $ x \in \reals ^ { N +1}$ instead
  of $ \complexes ^ { N + 1}$. This follows from 
the theorem of  Fubini. 

Next we observe that Barthe gives a description of the
family of $ L^ p$-spaces for which the form is finite in terms of a
family of inequalities satisfied by the reciprocals, $ ( 1/p_v ) _ { v
  \in E }$. It is known that the inequalities of Barthe describe the
matroid polytope, see the excellent monograph of J.~Lee
\cite{MR2036957}, for example. Finally, a multi-linear version of the
real method of interpolation allows us  to pass from $L ^ p $ estimates in
$ { \cal P }( E) $ to Lorentz space estimates in the interior of $ {
  \cal P } (E)$.  
See work of Christ \cite{MR766216} or  Janson
\cite{SJ:1986}  for the 
multi-linear interpolation results.  
\end{proof}

\note{ We will use $ \dot L^ 2 _ \alpha$ to denote the  homogeneous
  weighted space. We also use $ x^*$ for the conjugate,
  occasionally. } 

We will need to consider two sets of vectors in this section. The
first we will denote by $ E _ 1 \subset \reals ^ { N +1}$ for $ N
\geq 2$ and is given by 
\begin{equation} \label{E1def}
E _ 1 = E_1 ^ {N+1}  = \{ e _ 0, e_1, \dots, e_N, e_0-e_1, \dots , e _ { N -1} - e _N ,
\zeta = \sum _ { j = 0 } ^ N ( -1) ^ j e _ j  \} . 
\end{equation}
Note that the condition $ N \geq 2$ guarantees that $E_1$ contains $ 2N
+2$ distinct vectors.  The second set of vectors will only be needed
in odd dimensions. We define  
$ E_2 \subset \reals  ^ { 2N +1 }$, for $N \geq 1$ by 
\begin{multline}
\label{E2def}
 E _ 2= E _ 2 ^ { 2N+1} = \\
\{ e _ 0 , e_1 , \dots e _ { 2N } ,  \sum _ { j = 0 } ^ { 2k-
  1} ( - 1 ) ^ j e _ j, \sum _ { j = 0  }  ^ { 2k -1} (-1) ^ j  e_ {
  2N - j } ,   k = 1, \dots , N, \zeta = \sum _ { j = 0 } ^ { 2N } (
-1 ) ^ j e _ j \} .
\end{multline}

The following lemma  shows that the constant  function $ \theta (v) =
1/2$  lies in the interior of the  matroid polytopes for $ E_1$ and $
E_2$. 
\begin{lemma} \label{geom} For $j =1 $ or $2$,  let $ A _ { N+1 }= \{ \theta : \sum _{ v
    \in E _ j } |\theta (v) - 1/2 | \leq 1 , \ \sum   _ { v\in E _j }
  \theta (v) = N + 1 \} $. 
Then $ A _ { N + 1 } \subset { \cal P } ( E _1 ) $ for  $N = 2, 3,
\dots$ and $ A_ { 2N + 1 } \subset  { \cal P} (E _2)  $ for $ N = 1,2, \dots$. 
\end{lemma}
See Appendix \ref{matroids} for the proof. 

Our main goal for this section is to establish estimates for the
following two forms, 
$$
\Lambda_1 (q_0, \dots, q_N, t ) = \int _{\complexes ^ { N+1}} 
\frac { q_0 ( x_0 ) \dots q_N ( x_N) t( \zeta \cdot x) }
{ |x_ 0 - x_1 | \dots |x _ { N -1 } - x_N | } \, dx
$$
and
\begin{multline*}
\Lambda _ 2( q_0, \dots, q _ { 2N } , t )  = \\
 \int _ { \complexes ^ { 2N +1} } \frac { q_0 (x_0 ) \dots q _ { 2N } 
  ( x_ { 2N}) t ( \zeta \cdot x) } { |x_ 0-x_1|\dots |x_ 0 - x _
  1\dots - x _ { 2N -1 } | |x_ { 2N } - x _ { 2N-1} | \cdots | x _ {
    2N} -x_ { 2N -1} \cdots - x _1 | } \,d x.
\end{multline*}
The next lemma gives several estimates for these forms that are the
main step in obtaining estimates for the terms $ r_j$ in the expansion
(\ref{rexp}) of the scattering map. 

\begin{lemma} Suppose $ 0 \leq \alpha < 1$.  The following estimates for the forms $ \Lambda
  _1$ and $ \Lambda _2$ hold:
\begin{equation}\label{FormA}
\Lambda _ 1 ( q_ 0, \dots, q_{ 2N } , t ) 
\leq C \| t \| _ { H^ { 0, -\alpha} } \prod _ { j =0 } ^ { 2N} \| q _
j \| _ { H ^ { 0, \alpha / ( N+1)}},  \qquad N \geq 1
\end{equation}
\begin{equation}
\label{FormB}
\Lambda _ 1 ( q_ 0, \dots, q_{ N } , t ) 
\leq C \| t \| _ { L^2 } \prod _ { j =0 } ^ { N} \| q _
j \| _ { L^2}, \qquad N \geq 2
\end{equation}
\begin{equation}
\label{FormC}
\Lambda _ 2 ( q _0,  \dots, q_ { 2N} , t ) 
\leq C\| t \|_{ H ^ { 0, -\alpha} } \prod _ { j = 0 } ^ { 2N}  \| q _ j
  \| _ { H ^ { 0, \alpha / ( N +1) }}, \qquad N \geq 1.
\end{equation}
\end{lemma}

\begin{proof}
We begin by proving the estimate (\ref{FormA}) for the form $ \Lambda
_1$. We start with the elementary observation that for $ 0 \leq \alpha
\leq 1 $, we have 
$$
1 \leq \frac { |x_0 |^ \alpha + \sum _ { j = 1 } ^ N |x _ { 2j -1 } -
  x _ { 2j } |^ \alpha } { |x_0-x_1+ \dots +x_{2N} |^ \alpha } . 
$$
To establish a relation between $ \Lambda _1 $ and $ \Lambda( f_v |
v\in E_1) 
$, we define $ f_v ^ k  $ for $k  =0, \dots, N$ and $ v\in E_1$
by
\begin{align*} 
 &   f_ { e_0} ^ 0 = |\dotarg|^ \alpha |q_0| \\
&  f _ { e_j} ^ k   = | q_j|, \qquad & (j,k) \neq (0,0) \\
 &  f ^ k _ { e_ { 2k -1} - e_ { 2k } } =
1/ |\dotarg| ^ {     1-\alpha }, \qquad & k =1, \dots, N \\
&  f ^ k _ { e _ { j -1} - e _ j }  = 1 / |\dotarg|, \qquad & j \neq 2k \\
& f _ \zeta ^ k  = t / |\dotarg| ^ \alpha , \qquad &  k = 0,   \dots,
N.
\end{align*}
 With these definitions we have 
\begin{equation}\label{simple}
|\Lambda _1 ( q_ 0, \dots, q_ { 2N } , t) | \leq 
 \sum _ { k =0 } ^ N \Lambda ( f _v ^ k | v \in E _1) . 
\end{equation}
Apply Theorem \ref{forms} to the first term in the sum on the right-hand side of
(\ref{simple}) and observe that Lemma \ref{geom} tells us that the 
constant function $ \theta = 1/2$ is in the interior of $ { \cal P }
(E_1)$. Now we are able to conclude that 
$$
\Lambda ( f_ v ^ 0 | v\in E _1 ) \leq C\| t \| _ { \dot L ^ 2
    _ {-\alpha} }  \| q _ 0 \| _ { \dot L^2 _ \alpha }
  \prod _ { j =1 } ^ { 2N } \| q _ j \| _ { L ^ 2 } 
. 
$$
Here we have used that $ |\cdot| ^ { -1} $  is in the Lorentz space $
L^ { 2, \infty}$. For the terms $ \Lambda (f_v ^ k |v\in E_1)$, $k\geq
1$, we 
let $ 1/p _ { e _ 0 }= ( 1+ \alpha ) /2$, $ 1/ p _ { e _ { 2k -1} - e
  _ { 2k}} = ( 1- \alpha ) /2$ and use Theorem \ref{forms} to obtain
that 
\begin{multline}
\label{simpler}
\Lambda ( f _v ^ k | v \in E _1 )  \\ \leq C \| t \| _ { \dot L ^ 2 _ {-\alpha
}}
\| q _ 0 \| _ { L ^ { 2/ ( 1+ \alpha ) , 2 }} 
\left ( \prod _ { j =1 } ^ { 2N } \| q _ j \| _ 2 \right ) \| |\cdot |^ {
  \alpha -1 } \| _ { L ^ { 2 / ( 1- \alpha ) , \infty }} \cdot \|
|\cdot |^ { -1} \|^ { 2N -1} _ { L ^ {2, \infty } } .
\end{multline}
The generalization of  H\"older's inequality to Lorentz spaces (which
may be proven by   multilinear interpolation) implies that 
\begin{equation}  \label{holder}
\| q _ 0 \| _ { L ^ { 2 / ( 1+ \alpha ) ,  2 } } \leq C \| q _ 0 \| _
   {\dot L ^ 2 _ \alpha } \| | \cdot | ^ { -\alpha } \| _ {   L ^ { 2/
       \alpha , \infty } }  . 
\end{equation}
Since $ \| |\dotarg | ^ { -\alpha } \| _ { L ^ { 2/ \alpha , \infty }}\leq
C _\alpha$, from  (\ref{simpler})  and (\ref{holder}),
we may conclude that  
$$
\Lambda ( f _ v ^ k | v \in E _1 ) \leq C \| t \| _ {\dot  L ^ 2 _ {
    -\alpha } } \| q _ 0 \| _ { \dot L ^ 2 _ \alpha } \prod _ { j = 1} ^
        {2N}\| q _ j \| _ { L ^ 2} .
$$

A similar argument replacing $ 0$ by another even index  $2k$ implies
that
we have the estimates 
$$
\Lambda _ 1 ( q _ 0, \dots, q _ { 2N }, t )
\leq C \| t \| _ { \dot L ^ 2 _ { -\alpha } } \| q _ { 2k } \| _ {\dot L ^ 2 _
  \alpha } \prod _ {j \neq 2k } \| q _ j \| _ { L ^ 2 } , \qquad k =
1, \dots, N.
$$
Combining the cases $ \alpha =0$ and $ \alpha >0$ and  multi-linear
interpolation by the complex method gives the estimate 
(\ref{FormA}). 
The estimate (\ref{FormB}) follows  directly from  Theorem \ref{forms}
and Lemma \ref{geom}. 

Finally,  the estimate (\ref{FormC}) can be obtained by beginning with
the elementary estimate
$$
1 \leq \frac { |x_0 - x _1 + \dots - x _ { 2k -1} |^ \alpha + |x_ { 2k
  } |^ \alpha + | x_ { 2k+1} - x _ { 2k+2} \dots - x _ { 2N } | ^
  \alpha } { | x_ 0- x_1 + \dots + x _ { 2N } |^ \alpha } 
$$
and arguing as in the case of (\ref{FormA}). 
\end{proof}

We now turn to the estimates for the multi-linear expressions $ r_j$
defined in (\ref{multdef}). Using the estimates (\ref{FormA}) and
(\ref{FormC}) we prove the following result.

\begin{proposition} \label{term} If $ \alpha \in [0, 1) $ and $ r_j $
    is as defined in (\ref{multdef}), then 
\begin{align} \label{smooth} 
\| r _ j \| _ { H ^ { \alpha, 0 } } & \leq C_j \| q \| ^ { 2j +1 } _ { H ^
  { 0, \alpha / ( j+1)}},  
\\
\label{decay}
\| r _ j \| _ { H ^ {0,  \alpha } } & \leq C_j \| q \| ^ { 2j +1 } _ { H ^
  {\alpha / ( j+1), 0} }.
\end{align}
\end{proposition}

\begin{proof} Using the definition of $ r _j$, (\ref{multdef}) and the
  theorem of Fubini, we have
$$
\int _{ \complexes } r _j (k) t ( k) \, dk =  \frac 2 { (2\pi)^ { 2j
    + 1} } \int _ { \complexes ^ { 2j +1}} \frac { q ( y _0 ) \bar q
  (y_1)\dots q ( y _ { 2j } ) \hat t ( y_0- y _1 + y _2- \dots + y _ {
    2j } )  }
{ ( \bar y _ 0 - \bar y _ 1 ) ( y_1 - y _ 2) \dots ( y _ { 2j-1 } - y _
  { 2j }) } \, dy. 
$$
The use of Fubini can be justified since we assume that $q$ and $t$
are in the Schwartz class. From the above displayed equation, it is
easy to see that 
$$
|\int _ { \complexes }  r _j (k ) t (k) \, dk | \leq \frac 2 { ( 2\pi )
  ^ { 2j + 1}} \Lambda _1 (  |q|,\dots, |q|, |\hat t |) 
$$
where the form $ \Lambda _1 $ acts on $ 2j +1 $ copies of $|q|$. The
estimate (\ref{FormA}) and duality implies that for $ 0 \leq \alpha <
1$, 
$$
\| |D|^ \alpha r _ j \| _ { L^ 2 } \leq C ( \alpha, j) \| q \| _ { H ^
  {0, \alpha/(j+1) } }.
$$
Combining the cases $ \alpha =0$ and $ \alpha > 0$ gives the estimate
(\ref{smooth}).

To obtain the decay of $r_j$ in estimate (\ref{decay}), we again start
with the definition of $r_j$ in (\ref{multdef}) and  use (\ref{C1}) and
(\ref{C2}) alternately to obtain
\begin{multline*}
\int_{\complexes} t (k ) r _ j (k) \, dk 
= \frac { 2 }{ ( 2 \pi ) ^ { 2j +1} } 
\\
\times 
 \int _ {\complexes ^ { 2j+1} } 
\frac { \hat q ( k _ 0) \bar { \hat q }( k _1) \dots \hat q ( k- k_0+
  k_ 1 - \dots + k _ { 2j-1})
 t (k  )}{ ( \bar k _ 0 - \bar k
  _ 1 )\dots ( \bar k _0 - \bar k _ 1 + \dots - \bar k _
  {2j-1} )(k_0- k )\dots
( k_0 -k_1+ \dots + k_{2j-2} -k )  }
\\ \qquad  dk_0\dots dk_{2j-1}dk.
\end{multline*}
We make the change of variables $ k = k_0 - k _1+\dots + k_{2j}$, $dk
= dk_{2j}$ and obtain
\begin{multline*}
\int _{ \complexes} t(k) r _j (k) \, dk 
=  \frac { 2} { ( 2\pi ) ^ { 2j+1} } \\
\times  \int _ { \complexes ^ { 2j + 1}
} 
\frac { \hat q ( k _ 0) \bar { \hat q }( k _1) \dots \hat q ( k _ {
    2j } ) t ( k _ 0 - k_1+\dots + k _ { 2j })}{ ( \bar k _ 0 - \bar k
  _ 1 )\dots ( \bar k _0 - \bar k _ 1 + \dots - \bar k _
  {2j-1} )(  k _ { 2j } -  k _ { 2j -1 } )\dots
(  k _ { 2j } -  k _ { 2j -1 } + \dots  k _1 )} d\kappa. 
\end{multline*}
On the right, $d\kappa= dk_0\, dk_1 \dots dk_{2j}$. 
Now the estimate   (\ref{FormC}) quickly leads to the result
(\ref{decay}). 
\end{proof}

\section{The remainder term}
\label{best}

The final section gives  estimates for the remainder term. The
moral of this section is that when $ q \in H ^ { \epsilon , \epsilon }$
with $ \epsilon >0$, then the operator $ T^N_k(\mu_1) $  becomes smoother and decays more rapidly as $N$ increases. 
This allows us  to estimate the remainder 
 $r ^ { (N) } $ (defined in (\ref{remdef}) for $N$ large. 
The details are a bit  tedious. 
We begin by listing several  properties of the operator $T_k$ that start to
make the previous sentences precise.   

\begin{proposition} \label{mess}
Let $ q \in H^ { \epsilon, \epsilon}$ with $0< \epsilon < 1$ and suppose
that $\| q \| _ { H^ { \epsilon , \epsilon } } \leq M_0$.  
If $ \mu_1 = \mu _1(q; \cdot, \cdot)$ is the solution constructed in
Proposition \ref{MUEX}, 
we have the
following estimates for $ T _ k ^ { N } ( \mu_1) $.  Assume that
$p$ and $ \tilde p$ are related by $ 1/p = 1/2 + 1/ \tilde p$. 

a) We  have $ T _k ^ N ( \mu _1)  \in L^ {\tilde p} $ if $ 1/ \tilde p \in [0, \frac {
    \epsilon N } { 2 } ) \cap [0,  \frac 1 2 ) $ and 
\begin{equation} \label{hlsn}
\| T _ k ^ N ( \mu _1)  \| _ { L ^ {\tilde p} } \leq  C \| q \|^N _ { H ^ {
    \epsilon, \epsilon} } \| \mu _1 \| _ { L ^ \infty } . 
\end{equation}

b) For $ 1/ \tilde p \in ( 0, 1/2 )$ and $ j \geq 0$, we have 
\begin{equation}
\label{afrn} 
\| T_ k ^ { N + 2 j } ( \mu _1) \| _ { L ^ {\tilde p} } 
\leq C \langle k \rangle ^ { - j \epsilon } \| q \| _ { H ^ { \epsilon,
    \epsilon } } ^ { 2j }  \| T_k ^ N ( \mu _1 ) \| _ { L ^
{\tilde  p} }.
\end{equation}

\note{Conditions on $p$ below?}
c)  Provided $j \geq 0$ satisfies $ j\epsilon < 2 /p'$,  we have the estimate
\begin{equation} 
\label{cp}
\| T _ k ^ { N + j } ( \mu _ 1 ) \| _ { L ^ {\tilde p} _ { j \epsilon} } 
\leq C \| q \| _ { H ^ { \epsilon, \epsilon }}^ j \| T _ k ^ N ( \mu _
1) \| _ { L ^ {\tilde p} } .
\end{equation}

If in addition, $q ' \in  H^ { \epsilon, \epsilon }$ with $ \| q ' \| _ { H
  ^ { \epsilon, \epsilon } } \leq M_0$, we 
 have the following estimates for differences. 
Given $ 1/\tilde
p \in
[0, 1/2)$, there exist $N$ such that 

a')
$$
 \| T^N _ { k,q} ( \mu _1 ( q; \cdot, k ) ) - T_ { k,q' } ^ N ( \mu _1 (
 q'; \cdot, k ) ) \| _ { L ^ {\tilde p} } 
\leq C ( M _0 ) \|q - q' \|_ { H ^ { \epsilon, \epsilon } } .
$$

b') For $ 1/\tilde p$ in $(0,1/2)$ and $ j \geq 0$, we have 
$$
 \| T _ { k,q } ^ { N + 2j} ( \mu _ 1( q;\cdot ,k) ) - 
T^ { N + 2j }  _ { k,q'} ( \mu _ 1 ( q'; \cdot, k ) ) \| _ { L^
  {\tilde p} } 
\leq C(M_0) \langle k \rangle  ^ { -j\epsilon} \| q - q' \| _ { H ^ {
    \epsilon, \epsilon} } .
$$

c') Given $\tilde p$ in $[2, \infty)$ and $ j$ and $ \epsilon$ with $ 0 \leq
  j \epsilon < 2/ p'$, we have 
$$ \| T _ { k,q } ^ { N+j} ( \mu_ 1(q; \cdot, k )) - T _ { k, q' } ^
{ N + j } ( \mu _ 1 ( q'; \cdot, k) ) \| _ { L ^{\tilde p} _ { j \epsilon } } 
\leq C (M _0) \| q - q' \| _ { H ^ {\epsilon ,\epsilon }}. $$
\end{proposition}

\begin{proof} Note that by Proposition \ref{LPREG}, we have that $ \mu
  _1$ lies in $L^\infty ( \complexes^2)$. 
The first estimate is a consequence of the
  Hardy-Littlewood-Sobolev inequality (\ref{HLS}), the estimate
  (\ref{vekua}),  H\"older's
  inequality, and the observation  (\ref{embed}) that  $ H ^ {
    \epsilon, \epsilon} 
  \subset L^p$, if $ 1/p \in [ 1/2 - \epsilon/2, 1/2 + \epsilon /2 )
    $. Thus the maps 
$
f \rightarrow qf \rightarrow T_k(f) 
$
take  $ L ^ { \tilde p } \rightarrow L^ t \rightarrow L ^ { \tilde
  p_1}$   where $ 1/ t \in  [ 1/ \tilde p + 1/2 - \epsilon /2,   1/
  \tilde p  + 1/2 +  \epsilon /2 )\cap [ 0, 1/2)  $ 
and $  1/ \tilde p_1 \in [ 0, 1 / \tilde p  + \epsilon /2) \cap [
    0,1/2)$, and 
$$
\| T _k ( f ) \| _ { L ^ { \tilde p _1 }} \leq C\|q\|_{H^ {
    \epsilon, \epsilon }} \| f\| _ { L^ { \tilde p}}.
$$
Iterating this estimate gives a). 

The estimate b) follows quickly from (\ref{AFR}) and can be found in 
the work of Astala, Faraco, and Rogers 
\cite{MR3382582}. 

The third estimate (\ref{cp}) depends on the  result (\ref{SW}) on
fractional integration in weighted Lebesgue spaces.   
Using the estimate (\ref{SW}) and H\"older's inequality,  the maps  $ f \rightarrow qf \rightarrow T_k(f)$ will map $ L^ {
  \tilde p } _ \alpha \rightarrow L^ p _ { \alpha + \epsilon }
\rightarrow L ^ { \tilde p } _ { \alpha + \epsilon }$ where as usual $
1/p = 1/\tilde p + 1/2$ and the second step requires the condition
that $ - 2 / \tilde p < \alpha + \epsilon < 2 / p'$  in order to
use our result on fractional integration (\ref{SW}). 

The estimates for the differences follow by recognizing that each
term is a multi-linear expression in several copies of $q$ and $
\mu_1$, the continuity of $ \mu_1$ with respect to $q$
given in Proposition 
\ref{MUEX} and the result of Lemma \ref{mlllc}
\end{proof}

We begin by showing that if $ q \in H ^ { \epsilon , \epsilon }$ then
given  $ \alpha $ we may choose $ N = N ( \alpha) $ so that $ r ^ { ( N
  )} $ lies in $H ^ { 0 , \alpha}$.

\begin{lemma}
\label{rdl}
 Let $ q \in H ^ { \epsilon, \epsilon } $ with $0< \epsilon <1 $. Given $
  \alpha \in \reals$, there exist $N$ such that $ r ^ { (N) }$ lies
  in $ H ^ { 0, \alpha }$. If $ q $ and $ q '$  lie in $ \{ q \in H ^ {
    \epsilon, \epsilon } : \| q \| _ { H ^ { \epsilon , \epsilon } }
  \leq M _0\}$, then 
\begin{align} 
\label{remdecays} \| r ^ { (N) } \| _ { H ^ { 0 , \alpha }} & \leq C ( M_0, \alpha ) \\ 
\label{diffdecays} \| r ^ { (N) } ( q; \cdot ) - r ^ { (N) } ( q'; \cdot ) \| _ { H ^ {
    0, \alpha } } 
& \leq C( M_0, \alpha )  \| q - q ' \| _ { H ^ { \epsilon, \epsilon } }
  . 
\end{align}
\end{lemma}
\begin{proof} This is straightforward and follows an argument in
  Perry \cite{PP:2011a}. Since $ q \in H ^ { \epsilon, \epsilon }$,
  according to (\ref{embed}) we have $ q \in L^
  p$ with $p $ defined by $ 1/p = 1/2 + \epsilon / 4$, say.  We let $
  p'$ be the conjugate exponent as usual. According
  to parts a) and b) of  Proposition \ref{mess},  we may choose $N$ so
  that 
$$
\| T _ k ^ { 2N } ( \mu _1 ) \| _ { L ^  { p ' } } \leq C ( M_0,
\alpha ) \langle k \rangle ^ { -2 - \alpha } . 
$$
Thus by H\"older's inequality, we have 
$$
|r ^ { ( N)} (k) | \leq  \frac 1 \pi \int _ \complexes  | e_k (x) q(x) T _k ^ { 2N}
( \mu _1 ) |  \, dx 
\leq \frac 1 \pi \| q \| _ { L^ p } \| T _ k ^ { 2N } ( \mu _ 1) \| _
     { L ^ { p'}} 
\leq C  \langle k \rangle ^ { -2-\alpha } .
$$
This gives the first conclusion (\ref{remdecays}). 

To estimate the differences, we observe that   $ r^ { (N)}$ is a
multi-linear operator in $ q$, $ \bar q$ and $ \mu_1$ and use Lemma \ref{mlllc}
and Proposition \ref{MUEX}  to estimate $ \mu _1(q;\cdot) - \mu_1 ( q' ;
\cdot ) $ in terms of 
$ q- q'$.  With these observations the continuity (\ref{diffdecays})
follows from Lemma \ref{mlllc}. 

\end{proof}

We employ a similar strategy to estimate $ \|r \| _ { H ^ { \alpha , 0
}}$.  As part of this we will  need a lemma to show that $ r ^ { (N)
}$ is smooth. The  proof of this result is more  involved. 

To estimate $ r ^ { ( N) } (k) $, we begin by computing 
$$
\frac \partial { \partial k }   r ^ { (N)} (k)
=- \frac 1 \pi \int_\complexes y q(y) e_k (y)  \bar T_k ^ { 2N} ( \mu _1 ) (y) \,
dy 
+ \frac 1 \pi \int _ { \complexes } q ( y ) e _k (y) \frac \partial {
  \partial k } \bar T _ k ^ { 2N } ( \mu _1 ) ( y) \, dy. 
$$
Thus we need to study the expressions $ \frac \partial { \partial 
 \bar k } T _ k ^ { 2N } ( \mu _1 ) $ and $ \bar y T _ k ^ { 2N } ( \mu _1
)$. 

Our first step is to derive an expression for  the  $\partial/ \partial
  \bar k $-derivative of the function $ T _ k ^ { 2N } ( \mu _1)
$. This generalizes  the $ \partial / \partial \bar k $
equation for $ \mu _1$. For this exercise, we assume that $q$
is  in the Schwartz space, which will allow us to concentrate on the
formulae rather than convergence. 

\begin{lemma} 
\label{t2nd} 
If $ q$ is in $ {\cal S } ( \complexes)$,  then 
$$
\frac \partial { \partial \bar k }  T_ k ^ { 2N } ( \mu _ 1) 
= \frac 1 2 \bar r (k ) T _ k ^ { 2N } ( \mu _2) 
+ \frac 1 { 2 \pi } \sum _ { j = 1} ^ N T _ k ^ { 2j -1} ( 1) \int _
{\complexes} \bar q ( x) e _ k ( - x)   T _ k  ^ { 2N - 2j} ( \mu _ 1
) (x) \, dx.
$$
\end{lemma}
\begin{proof}
Taking the derivative with respect to $ \bar k$ gives 
\begin{align*}
&\frac \partial { \partial \bar k } T _ k ^ { 2N } ( \mu _1 ) (x)  \\
&= T _ k ^ { 2N } ( \frac { \partial \mu _ 1 } { \partial \bar k }) (x) 
+ \frac 1 { 2\pi } \sum _ { j = 1 } ^ N T _ k ^ { 2j - 1} ( 1) (x)
\cdot \int _ { \complexes } \bar q ( x _ { 2j } ) e _ k ( -x_ { 2j } )
T _ k   ^ { 2N - 2j } ( \mu _1 )(x) \, dx _ { 2j } .
\end{align*}
If we recall the $ \partial/ \partial \bar k $-equation  for $ \mu$ (see Perry
\cite{PP:2011a}, for example), $ \frac { \partial } { \partial \bar k } \mu _1 = \frac
1 2 \bar r \mu _ 2$, we obtain the Lemma. 
\end{proof}

\begin{lemma} 
\label{t2ny} 
If $ q\in { \cal S } ( \complexes)$ and $ 1 \leq n \leq N$, then 
\begin{align*} 
&\bar x T _ k ^ N (f) (x) \\ &=  \eta (x) (\sum _ { j = 1} ^ n  \tilde T _ k ^ { j
  -1} (1)(x) ( \frac 1 { 2\pi }  \int _ { \complexes } q (y) e _ k( y)
T _k  ^{ N- j}(f) (y) \, dy )  ^{ *j-1 } 
+  \tilde T _ k ^ n ( ( \bar \cdot ) T _ k ^ { N -n } (f) )
(x)) .
\end{align*}
Here $ \eta (x) = \bar x / x $ and $ \tilde T_k  = T _ { k , \tilde q
} $  and $ \tilde q = \eta q$.  We use $ x^{*j}$ to denote $j$
applications of the map $ x \rightarrow \bar x$. Thus $ x^ {*j} = x$ if $
j $ is even  and $ x^ {*j}= \bar x $ if $ j $ is odd. 
\end{lemma}

\begin{proof}
We begin with the identity, 
\begin{align*}
\bar x T _ k (f) (x) & = \frac { \bar x } { x}   \frac 1 { 2 \pi }
\int _ { \complexes } q ( y ) e _ k (y )  \bar f ( y ) \, dy  
+ \frac {\bar x} x  \frac 1 { 2 \pi } \int _ { \complexes } \frac { q ( y ) e _ k (y ) y
  \bar f (y) } {x- y} \, dy  \\
&= \frac { \bar x } { x}   \frac 1 { 2 \pi } \int _ { \complexes }  q
(y) e _ k ( y ) \bar f (y) \, dy  + \frac {\bar x} x \tilde T _k (
(\bar\cdot )  f ) (x) . 
\end{align*}
Iterating this result gives the Lemma. 
\end{proof}

We are ready to give an estimate on the smoothness of the function $
r ^ { (N) }$. 
\begin{lemma}
\label{smoothpart} 
 Let $ q \in H ^ { \epsilon, \epsilon }$ with $ \|
  q \| _ { H  ^ { \epsilon , \epsilon } } \leq M _0$. Then there exist
  $N  = N ( \epsilon )$ such that 
$$
\| \frac \partial { \partial  k }  r ^ { (N) } \| _ { L ^ 2 } \leq
C ( M _0).
$$
If $ q, q' $ are both in $ \{ q : \| q \| _ { H ^ { \epsilon ,
    \epsilon }} \leq M _0 \}$, then  
$$
\|\frac \partial {\partial  k }  r ^ { (N) }  ( q; \cdot ) -
\frac \partial { \partial  k }  r ^ { (N) } (q'; \cdot ) \| _ { L ^ 2
} 
\leq C ( M _0, \epsilon ) \| q - q ' \| _ { H ^ { \epsilon, \epsilon }
} . 
$$
\end{lemma}
\begin{proof}
We differentiate $  r ^ { (N) } $ and obtain 
\begin{align*}
\frac { \partial } { \partial  k } r ^ { (N) } (k) & = -\frac 1 \pi
\int _ { \complexes } q ( y) e _ k (y) y \bar T _ k ^ { 2N } ( \mu _1
) ( y ) \, dy 
+ \frac 1 \pi \int _ {\complexes} q (y) e _ k (y)  \frac \partial { \partial  k }
\bar T _k ^ { 2N } ( \mu _1 ) (y) 
\,dy  \\
&= \sum _ { j =1 } ^ N  I _j + II + \sum _ { j =1 } ^ N  III_j  + IV
\end{align*}
with 
\begin{multline*}
I _ j = - \frac 1 \pi \int _ { \complexes }  q ( y  ) \bar \eta (y) 
e _ k (y ) ( \tilde T _ k ^ { j -1 } ( 1 ) (y ) ) ^ * \, dy \\ 
\times   (
\frac 1 { 2 \pi } \int _ \complexes  q (y ) e _ k ( y ) T _ k ^ { 2N - j } ( \mu _1
) ( y) \, dy ) ^ { *j }  = A_j B _j
\end{multline*}
$$ 
II = - \frac 1 \pi \int _ { \complexes } q (y) \bar \eta (y) e _k (y)
( \tilde T _ k ^ N (( \bar \dotarg  ) T_k ^ N ( \mu _1)))^ * (y) \, dy. 
$$
$$
III_j = \frac 1 { 2 \pi ^ 2 } \int_\complexes q (y) e _k (y)  ( T_k ^ { 2j -1 }
(1) (y) ) ^ * \, dy \cdot \int _{\complexes} \bar q (y )  e _ k ( y ) T _ k ^
{ 2N - 2j } ( \mu _ 1) ( y ) \, dy  = C_j D _j .  
$$
 Finally,  we define $ IV$ by 
$$
IV =  \frac 1 { 2\pi }  r ( k ) \int_\complexes q (y) e _k (y ) ( T_k ^ { 2N}(
\mu _2 ) ( y ) ) ^ * \, dy . 
$$
To obtain this representation we use  
 Lemma \ref{t2nd} and then 
 Lemma \ref{t2ny} to commute $y$ through $ N $
applications of $ T_k$.

We proceed to show that each of the terms  $I _j$, $II$, $III_j$, and
$IV$ are in $L ^ 2$. 
To estimate the terms $ I _ j = A_j \cdot B_j$, we first suppose that
$ j \geq 3$. Then estimate (\ref{FormA}) or (\ref{FormB}) gives that $
A _ j $ is in $L^2$. Since $ 2N - j \geq N$, we may choose $N$ large
so that part a) of Proposition \ref{mess} gives $ T _ k ^ { 2N -j } (
\mu _1 ) \in L^{p'}$ for $ 1/ p ' = 1 /2 - \epsilon /4$, say.
Since we also have    $q \in L^p$, 
we conclude  that $ B_j $ is in $L^ \infty$. For the case $j =2$, we use Lemma
\ref{A2} to conclude that  $A_2$ is in $L^p$ for some $p > 2$ while by
(\ref{AFR}), we may choose $N=N(\epsilon)$ so that  $|B_2 (k) | \leq
\langle k \rangle ^ { -2} $ and thus the product $ A_2\cdot B _2$ is
in $L^2$. Finally, $ A_1 = \frac 1 2 ( q \bar \eta ) \hat {} $ and
hence lies in $L^2 $ by Plancherel's theorem.  

For the term $II$, we begin by using (\ref{AFR}) and Proposition
\ref{LPREG}  to
conclude that for $N$ large and $ 1 / \tilde p = \epsilon /4$, say,
that 
$$
\| \langle \cdot \rangle ^ { 1- \epsilon } T _k ^ N ( \mu _1 ) \| _ {
  L^ { \tilde p }} \leq C ( \| q \| _ { H ^ { \epsilon, \epsilon }} )
\langle k \rangle ^ { -2} .
$$
Then since $ \langle \cdot \rangle ^ \epsilon  q \in L^ 2$, we have 
$ \|(\bar \cdot)  q  T _k ^ N ( \mu _ 1 ) \| _ { L ^ p } \leq C (
\| q \| _ { H ^ { \epsilon, \epsilon }}) \langle k \rangle ^ { -2} $
for $ 1 /p = 1/2 + \epsilon / 4$. Now applying $\tilde T _k ^ N$ (and
possibly choosing $N$ larger) we have $ \| \tilde T _k ^ N ( ( \bar
\cdot ) T _k ^ N ( \mu _ 1 )) \| _ { L ^ { \tilde p _ 1}} \leq C ( \| q
\| _ { H ^ { \epsilon, \epsilon }} ) \langle k \rangle ^ { -2} $ for $
1/ \tilde p _ 1 = 1 /2 + \epsilon / 4$. Then  H\"older's inequality
gives 
the estimate $ |II(k) | \leq \langle k \rangle ^ { -2} $ and hence
that $II$ lies in $L^2$. A similar strategy handles the term $ III_j =
C_j \cdot D_j$ for $ 2j \leq N$. Here $C_j$ is  in $L^2$   by
(\ref{FormA}) as long as $2j\geq 4$. The estimate (\ref{hlsn}),
estimate (\ref{embed}),  and H\"older's inequality imply that $D_j$ is
bounded.  When $ 2j=2$, we use Lemma \ref{A2} and (\ref{AFR}) as in
the estimate for $I_2$.

Next we consider $ III_j$ when $ 2j \geq N$. Here we argue as in Lemma
\ref{rdl} 
to conclude that $ D_j$ is in $L^2$ and since $ 2j \geq N$, we may
choose $N$ large so that (\ref{hlsn}) and H\"older's inequality imply
$C_j$ is bounded. 

Finally, we show the term $IV$ is in $L^2$. We first recall that
Astala, Faraco, and Rogers \cite[Theorem 3.4]{MR3382582}  show that the map $ q
\rightarrow {\cal R}(q)$ maps $H^ {
  \epsilon, \epsilon}$ to $L^2$ and is locally Lipschitz
continuous. With this result, we only need to 
show that the map 
\begin{equation} \label{Done} q \rightarrow  \int_\complexes q (y) e _k (y ) ( T_k ^ { 2N}(
\mu _2 ) ( y ) ) ^ * \, dy 
\end{equation}
takes $H^ { \epsilon , \epsilon}$ into $L^ \infty$. From the estimate
(\ref{embed}), we have that $ q \in L^p$ for some $p >2$ while part a)
of Proposition \ref{mess} and (\ref{IE1B}) imply  that for $N$ large, we have $T_k ^
{ 2N} (\mu_ 2 ) $ in the dual space, $L^{p'}$.  It follows that 
$\int_\complexes q (y) e _k (y ) ( T_k ^ { 2N}(
\mu _2 ) ( y ) ) ^ * \, dy$  lies in $L^\infty$.  

To complete the proof, we need to show that each of these terms is a
locally Lipschitz continuous function of $q$. This follows from
similar arguments using Proposition \ref{mess} a'-c') and Lemma
\ref{mlllc}. 
\end{proof}

\begin{lemma}\label{A2} If $ q \in L^ { s _ q}$ with $ 4/3 < s _q< 2$,
  then $ r (k) = \int _{\complexes } q (y) e _ k (y) T _ k (1) (y) \,
  dy $ lies in $ L^ { s _ r}$ where $ 1/s_r = 3/2 - 2 / s _q$ and we
  have the estimate
$$
\| r\| _ { L ^ { s_r} } \leq C \| q \| _ { L^ { s _ q }} ^2.
$$
\end{lemma}

\begin{proof} We will estimate $ r$ by duality and thus we choose $ t
  $ a nice function in $ L^ { s'_r}$. We write 
\begin{equation}\label{form2}
\int _{ \complexes } t (k ) r (k) \, dk 
= \frac 1 { 2 \pi ^2}\int _ {\complexes} \frac { \bar q (z) q(y) \hat
  t ( y-z) } { \bar y - \bar z } \, dy dz.
\end{equation}
We may use the Hausdorff-Young inequality and  the interpolation result of Christ \cite{MR766216} or  Janson \cite{SJ:1986} to find
that the right-hand side of  (\ref{form2}) is finite when  $
2/s_q + 1/s_r = 3/2$. Note that the condition on the second index in
the Lorentz spaces will always hold true since  $ 2/s_q \geq
2$. 
\end{proof}

We are ready to give the proof of Theorem \ref{scattered}. 
\begin{proof}[Proof of Theorem \ref{scattered}]
We write $ r (k) = {\cal R } (q) (k)  = \sum _ { j = 0 } ^ N r _ j (k)
+ r^ { (N) } (k)
$ where $ r_j$ and $ r ^ { ( N )} $ are as in (\ref{multdef}) and (\ref{remdef}).  According to
Proposition \ref{term} 
 we have  $r _j$ is in $H ^ { \beta, \alpha }$ if $ q $ is
in $H ^ { \alpha, \beta }$. 
By Lemma \ref{rdl} and Lemma \ref{smoothpart}, given $ \epsilon > 0$,
we may find 
$N=N(\epsilon)$ so that 
 $ r ^ { (N)} $ is in $ H ^ {  \beta, \alpha }$  if $ q $ is in $ H ^
{ \epsilon, \epsilon }$. 

To estimate $ r - \hat q$, we note that $ r_0 = \hat q$ and then use
Lemma \ref{forms} to conclude that  $  r _ j  \in  H ^ { 2 \beta,
  2\alpha }$ when $ q \in H ^ {\alpha, \beta}$ and  $j \geq 1$. Finally, we use Lemma \ref{rdl}  and  Lemma \ref{smoothpart} to conclude that the remainder  $ r^ {(N)}$ lies in $ H^ { 2\beta, 2\alpha}$.  This gives the second estimate of the
  Theorem. 
\end{proof}

\appendix
\section{Appendix: The matroid polytope} 
\renewcommand{\theequation}{\Alph{section}.\arabic{equation}}
\renewcommand{\thetheorem}{\Alph{section}.\arabic{equation}}

\begin{center} 

\parbox {2in}{Nathan Serpico \\ Department of Mathematics \\ 
University of Kentucky \\
Lexington KY 40506-0027} \qquad \parbox {2in}{Russell M. Brown 
\\ Department of Mathematics \\ 
University of Kentucky \\
Lexington KY 40506-0027} 

\end{center}

\label{matroids} 
We give the proof of Lemma \ref{geom} in this section. 
We let $E$ denote one of the matroids $ E_1$ or $E_2$ defined in
section \ref{bl}. 
Our strategy is to show that for any pair of vectors $(v,w) $ from  $E$, we may
find two bases $ B_1$ and $ B_2$ so that $ \{v\} = B_1 \cap B_2$ and $
\{ w \} = E \setminus ( B_1 \cup B _2)$.  Since  $ \chi _ { B_j}
\in {\cal P } (E)$ and $ {\cal P}(E)$ is convex,   it follows  that $ \Phi _ {v,w} = \frac  1 2 (
\chi _{B_1} + \chi _ { B_2}) $ lies in $ { \cal P }(E)$. We have that 
$$
\Phi _ { v,w}  ( u ) = \left \{ \begin{array}{ll} 
1 , \qquad & u = v\\
0 , \qquad &u = w \\
1/2, \qquad& \mbox{else}. 
\end{array}
\right. 
$$
It is not hard to show that $ \cal A $, the set in Lemma \ref{geom},  is the convex hull of $ \Phi _{
  v,w }$ for all pairs $(v,w)$ from $E$. 

\note{Compare the value of $N$ in case 5. } 
We first consider the matroid $E_1$ with $N \geq 2$. 
\begin{lemma} Consider a pair $(v,w) $ with elements from $E_1$. We may
  find two bases $B_1$ and $B_2$  of $ \reals ^ { N +1}$ so that $  B
  _1 \cap B_2 = \{ v\}$ and  $E_1 \setminus ( B_1 \cup B_2 ) = \{ w \}$.

\end{lemma}

\begin{proof} Our proof is not particularly clever. We consider a
  number of 
  cases and list the bases in each case.  In several of the cases, we
  will make use of the map $ e _ k \rightarrow e _ { N -k } $ which preserves the elements of $E_1$, at  least up to a sign. This will allow us to assume that the vectors  satisfy certain extra conditions. We let $ S = \{ e _0, e_1,
  \dots , e_N\}$,  $D = \{ e _0 -e _1, \dots, e_{N-1}-e _N\}$, and $
  \zeta = e _ 0 - e_1+ \dots + ( -1)^Ne _N$. 

{\em Case 1: }  $ v = e _k $ and $ w = e_\ell$. In this
case we may let $ B _ 1 = (S\setminus \{ e _ \ell \} ) \cup \{ \zeta
\}$ and $ B _ 2 = D \cup \{ e _ k \}$. 

{\em Case 2: }  $ v = e _ k $ and $ w = e_ { \ell -1 } -e _
\ell$. We choose $ \hat e = e_N$ if $ k \leq \ell-1$ or $ \hat e = e
_0 $ if $ k \geq \ell$. Then we let  $ B_ 1 = ( S \cup \{ \zeta \} )
\setminus  \{ \hat e \} $ and $ B _ 2 = ( D \setminus \{ e _ { \ell -1
} - e_ \ell  \} )\cup \{ e _k , \hat e \}$.  

{\em Case 3: }  $ v = e _ { \ell -1 } - e _ \ell $ and $ w = e _k$.
{Subcase A: } $k \neq \ell -1$ or $\ell$.  Then  $ B_1 = \{ \zeta , e
_ { \ell -1 } - e _ \ell \} \cup ( S \setminus \{ e_k, e _ { \ell -1 }\})$
and let $ B_2 = D \cup \{ e _ { \ell -1 }\}$. 

Subcase B: If $ k = \ell -1 $ or $ \ell$, then choose $\hat e= e_m$ with $ 
m \neq \ell-1 $ or $ \ell$ and set $ B_1 = \{ \zeta , e _ { \ell -1 } -
e_ \ell \} \cup (S \setminus \{ \hat e , e _k\})$ and  $ B_2 = D \cup
\{ \hat e \}$. 

{\em Case 4: } $ v = e _ { k-1 } - e_k$ and $ w = e _ { \ell -1 } - e _ \ell
$. In this case we may flip the order and assume $ k < \ell$ if
necessary. We let $ B _1 = \{ \zeta , e _ { k -1 } - e _k \} \cup ( S
\setminus \{ e _ { k -1 } , e_ \ell \}) $ and $ B _2 = (D \setminus \{
e _ {\ell  -1 } - e _ \ell \}) \cup ( \{ e _ { k -1 } ,e _ \ell \})$. 

{\em Case 5: }$ v = \zeta $ and $ w = e _k$. In this case, we may
assume, 
after possibly flipping the order of the coordinates,  that $ k < N-1$, provided 
 that $ N \geq 3$. 
In this case, we let $ B_1 = \{ e _ 0-e_1, \dots, e _k - e _ { k +1 },
e _ { k+1 }, \dots, e _ { N -1 } , \zeta \}$ and $B _ 2 = \{ e _ 0 ,
\dots ,  e _ { k -1}, e _N, e _ { k +1 } - e_ { k +2} , \dots , e_{ N -1 }
- e_N, \zeta \} $.   The case $k =1$ and $N=2$ must be handled
separately and in this case we can let $ B_1 = \{ e_0, e_2, \zeta\}$
and $ B_2 = \{ e _ 0-e_1, e_1-e_2, \zeta\}$.

{\em Case 6: } $ v = \zeta $ and $ w = e _ k - e _ { k +1} $. Again,
we may assume that $ k < N-1$. We let $ B _1 = \{ e _ 1, \dots , e _ {
  k +1 } ,  e _ { k+1}- e _ { k+2} , \dots, e _ { N-1} -e _N , \zeta
\}$ and $ B _2 = \{ e _ 0, e _ { k +2 },\dots , e _N , e _ 0-e _1,
\dots, e _ { k -1 } -e _ k, \zeta \}$.  Again, the case $k =1$ and
$N=2$ must be handled separately and for this case, we can let $ B _1 = \{
e_0, e _2, \zeta\}$ and $ B_2 = \{ e _ 1, e_0- e_1, \zeta\}$. 

{\em Case 7: } $ v = e _ { k -1 } - e _k  $ and $ w = \zeta $. In this
case we  let $B_ 1 = ( S \setminus \{ e _k \} ) \cup \{ e _ { k -1 } -
e _k \} $ and $ B _2 = D \cup \{ e _k \}$.

\end{proof}

We let  $E_2$ denote the set of vectors  defined for $N \geq 1$  before Lemma
\ref{geom} and establish a similar result. 

\begin{lemma} Consider a pair  of vectors $ (v, w  ) $ from $E_2$.  We
  may find a pair of bases $ B_j$, $j=1,2$ so that $ B_ 1 \cap B_2 =
  \{ v\}$ and $ (E_2 \setminus B_1) \cap (E_2 \setminus B_2 )= \{
  w\}$. 
\end{lemma}

\begin{proof} Again the proof is by cases. We let $ S = \{ e_0, e_1,
  e_2, \dots e _ { 2N} \}$, $ D = \{ e _ 0-e_1, \dots, e_0-e_1+\dots -
  e _{ 2N-1}, e_{2N}-e_{2N-1}, \dots , e_{2N}-e_{2N-1}+\dots -e _1 \}$, and set $ \zeta = e _0- e_ 1 + \dots + e_{ 2N}$. 

{\em Case 1: } $v = e_j $ and $ w = e _k$.    
Let $ B_ 1 =( S \setminus \{ e_k \} ) \cup \{ \zeta \} $ and $ B_2
= D \cup \{ e_j \}$. 

{\em Case 2: } $ v = e _j $ and $ w = \zeta$.
Let $ B_1 = S$ and $ B_2 = D \cup \{ e _j \}$. 

{\em Case 3:  } $v = \zeta $ and $ w = e _j$.
Let $B_1 = (S\setminus\{ e _j\} )  \cup \{ \zeta \} $ and $B_2 = D \cup \{
\zeta\}$.

{\em Case 4: } $ v = e _j$ and $ w =d \in D$.  If necessary, we may
apply the 
transformation, $ e _k \rightarrow e _ { 2N-k} $ so that the vector
$d$ has the form $ d= e_0-e _1+ \dots -e _k$. Now we choose a vector
$e _ \ell$ which depends on $ e_j$ with $ \ell$ defined as follows:
$$ 
e_\ell = \left \{ \begin{array}{ll} e_{2N} , \qquad &  j = 0 , \\ 
e_0, \qquad & \mbox{else}.
\end{array}\right .
$$
Then we let the bases be $ B _1 = ( S \setminus \{ e _ \ell \}) \cup
\{ \zeta \}$ and $ B _ 2 = (D\setminus \{ d \} ) \cup \{ e _j , e_ \ell
\}$. 

{\em Case 5: } $ v = d\in D$ and $ w = e _j $. 
We may assume that $ d = e_0-e_1+ \dots - e _k$, as above.  We choose
a vector $ e_\ell$ to be $ e_\ell = e_{ 2N}$ if $ j \leq k$ and $ e_
\ell = e _0 $ if $ j > k$. Then we put $ B_1 = ( S \setminus \{ e _j ,
e_\ell \}) \cup \{ \zeta , d\}$ and $ B_2 = D \cup \{ e_ \ell \} $. 

{\em Case 6: }  $ v = d _1 \in D$ and $ w = d _2 \in D$.
If necessary, we may apply the transformation,  $ e_j \rightarrow e _{
  2N-j }$ 
and assume that 
$ d_2 = e _ 0-e _1+\dots  -e_k$. We put $ B_1 = (S \setminus \{ e _0, e_{2N} \}
) \cup \{ \zeta, d_1 \} $ and $B _2 = ( D \setminus \{ d _2 \} ) \cup
\{ e_0, e_{ 2N}\} $.

{\em Case 7: } $v = d\in D$ and $ w = \zeta $.    By flipping the
vectors, we may assume that $d= e_0 - \dots - e _k$.  
 We put $ B_ 1 = (S \setminus \{ e_k \} ) \cup \{ d\}$ and  $
 B_2 = D \cup \{ e_k \}$.  

{\em Case 8: } $ v= \zeta $ and $ w = d \in D$. Again, we may assume
that $ d = e _0 -e _1 +\dots - e _k$. We let $ B_1 = (S \setminus \{ e
_ k \} ) \cup \{ \zeta \}$ and $ B _2 =(  D\setminus \{ d \} ) \cup \{ \zeta \} \cup \{ e _
k \}$. 

\end{proof}


\def\cprime{$'$} \def\cprime{$'$} \def\cprime{$'$} \def\cprime{$'$}
  \def\cprime{$'$} \def\cprime{$'$} \def\cprime{$'$} \def\cprime{$'$}
  \def\cprime{$'$} \def\cprime{$'$}

\noindent \small \today

\end{document}